\documentclass[12pt, a4paper]{amsart}

\usepackage{latexsym}

\usepackage{color}

\usepackage{amssymb}
\usepackage[centertags]{amsmath}
\usepackage{verbatim}

\usepackage{bbm}
\usepackage{amsthm}
\usepackage{graphicx}
\usepackage{float}

\usepackage[colorlinks=true ,linkcolor=black ,citecolor=black]{hyperref}

\usepackage[a4paper, hmargin=2.5cm, vmargin={3cm, 3cm}]{geometry}
\usepackage{amsthm}

\newtheorem{theorem}{Theorem}

\newtheorem{lemma}[equation]{Lemma}

\newtheorem{proposition}[equation]{Proposition}

\newtheorem{claim}[equation]{Claim}

\theoremstyle{definition}

\newcommand{\theoremname}{testing}
\theoremstyle{remark}
\newtheorem*{remark*}{Remark}

\numberwithin{equation}{section}

\renewcommand{\phi}{\varphi}

\newcommand{\eps}{\varepsilon}

\newcommand{\cA}{\mathcal A}

\newcommand{\bE}{\mathbb E}

\newcommand*{\dd}[1]{{\rm d}#1}

\newcommand*{\what}[1]{\widehat{#1}}

\def\re#1{\mathrm{Re}\,#1}

\renewcommand{\le}{\leqslant}
\renewcommand{\ge}{\geqslant}

\newcommand{\bR}{\mathbb R}
\newcommand{\bC}{\mathbb C}
\newcommand{\bZ}{\mathbb Z}
\newcommand{\bD}{\mathbb D}

\newcommand{\bP}{\mathbb P}

\newcommand{\ti}{\widetilde}

\title[Roots of random Littlewood polynomials]{Approximately half of the roots of a random Littlewood polynomial are inside the disk}

\author{Oren Yakir}
\address{School of Mathematical Sciences, Tel Aviv University, Tel Aviv 69978, Israel}
\email{oren.yakir@gmail.com}

\begin{document}
	\maketitle
	\begin{abstract}
		We prove that for large $n$, all but $o(2^{n})$ polynomials of the form $P(z) = \sum_{k=0}^{n-1}\pm z^k$ have $n/2 + o(n)$ roots inside the unit disk. This solves a problem from Hayman's book \cite[Problem~4.15]{hayman}.
	\end{abstract}

	\section{Introduction}
	\let\thefootnote\relax\footnote{Research supported by Israel Science Foundation Grants 382/15, 1903/18 and by European Research Council Advanced Grant 692616.}
	A \emph{Littlewood polynomial} is a polynomial with coefficients in $\{-1,1\}$. Let $n\ge 2$ be large and consider the random polynomial given by
	\begin{equation}
		\label{eq:def_of_polynomial}
		P(z):= \sum_{k=0}^{n-1} X_k z^k 
	\end{equation}
	where the $X_k$'s are independent random variables with 
	$$
	\bP\left(X_k = 1\right) = \bP\left(X_k = - 1\right) = \frac{1}{2}.
	$$
	So $P$ is a polynomial chosen uniformly at random among all $2^{n}$ Littlewood polynomials of degree $n-1$. It is well known that as the degree $n$ tends to infinity, the roots of $P$ tend to cluster uniformly around the unit circle, see \cite{sparo_sur} for the classical proof of this fact or \cite{ibragimov_zeitouni} for finer results and additional references. Denote 
	by $$\nu_n := \sum_{z: P(z) = 0} \delta_z$$ the (random) counting measure of roots of $P$ and by $\bD$ the unit disk in the plane. The following problem was posed in Hayman's classical problem book: \emph{``If $\eps_k = \pm 1$, is it true that, for large $n$, all but $o(2^n)$ polynomials of the form $P(z)= \sum_{k=1}^n \eps_k z^k$ have just $n/2 + o(n)$ roots in $\bD$?"}. In the recent fiftieth anniversary reprint \cite[Problem~4.15]{hayman}, it is stated that no progress on this problem has been reported. The same question was also asked in a paper by Borwien, Choi, Ferguson and Jankauskas \cite{borwein_choi_ferguson_jankauskas} which speculated the answer should be positive. Our main result is an affirmative answer to this question.
	\begin{theorem}
		\label{thm:main_theorem}
		We have
		\[
		\lim_{n\to\infty} \bP\left(\left|\nu_n(\bD) - \frac{n}{2} \right| \ge n^{9/10} \right) = 0.
		\]
		In particular, $\nu_n(\bD)/n \xrightarrow{p} 1/2$ as $n\to\infty$, where $\xrightarrow{p}$ denotes convergence in probability.
	\end{theorem}
	The exponent $9/10$ in Theorem \ref{thm:main_theorem} is not optimal and probably the correct order of the deviations is $\sqrt{n}$ , but the methods of this paper are insufficient to establish this. It may be that similar methods as in \cite{do_nguyen} can give the first order term in $\text{Var}(\nu_{n}(\bD))$, but we did not pursue this direction.
	
	The key step towards proving Theorem \ref{thm:main_theorem} is a concentration result for the logarithmic integral of $P$ and to describe it we fix some notation. Throughout we denote by $\dd\mu(x) := \dd x/(2\pi)$ the normalized measure on the interval $[-\pi,\pi]$. For $\theta\in[-\pi,\pi]$ fixed and for all $r>0$ we set
	\begin{equation}
	\label{eq:def_radial_variance}
	\sigma(r)^2 := \bE\left|P(re^{i\theta})\right|^2 = \sum_{k=0}^{n-1}r^{2k} = \frac{r^{2n} - 1}{r^2 - 1}, \qquad \left(\sigma(1)^2 = n\right)
	\end{equation}
	and $\ti{P}(re^{i\theta}):= P(re^{i\theta})/\sigma(r)$. Here and everywhere $\bE$ stands for the expectation with respect to the random coefficients of $P$. The following lemma is the key step towards proving Theorem \ref{thm:main_theorem}.
	\begin{lemma}
		\label{lemma:mean_of_log_mahler_measure}
		Let $\ell \in \{1,2\}$. For all $r \in [1-n^{-11/10},1+n^{-11/10}]$ we have that, as $n\to\infty$,
		\[
		\bE \left[\left( \int_{-\pi}^{\pi} \log|\ti{P}(re^{i\theta})| \dd \mu(\theta)\right)^{\ell} \right] = \left(-\frac{\gamma}{2}\right)^{\ell} + \mathcal{O}\left(\frac{(\log n)^2}{\sqrt{n}}\right),
		\]
		where $\gamma \approx 0.5772...$ is Euler's constant. 
	\end{lemma}
	
	Recall that for a monic polynomial $F(z) = \prod_{j=1}^{n} (z-\alpha_j)$ the \emph{Mahler measure} of $F$ is given by
	\begin{equation}
	\label{eq:def_mahler_measure}
	M(F) := \prod_{j=1}^{n} \max\{1,|\alpha_j|\}.
	\end{equation}
	Using Jensen's formula and \eqref{eq:def_mahler_measure} it is evident that
	\begin{equation}
		\label{eq:def_mahler_measure_logaritmic_integral}
		\log M(F) = \int_{-\pi}^{\pi} \log|F(e^{i\theta})|\, \dd \mu(\theta).
	\end{equation}
	With the above definition in mind, Lemma \ref{lemma:mean_of_log_mahler_measure} (the case $r=1$) tells us that the Mahler measure of a random Littlewood polynomials $P$ is typically close to its mean. The main technical difficulty in proving Lemma \ref{lemma:mean_of_log_mahler_measure} is the singularity of the logarithm near zero. For instance, $P(1)$ can vanish with positive probability if $n$ is even, and hence $|\bE\log|P(1)|| = \infty$. 
	
	In \cite[Theorem~1.2]{choi_erdelyi} Choi and Erd\'{e}lyi prove that
	\begin{equation}
		\label{eq:limit_of_mahler_measure_truncation}
		\lim_{n\to\infty} \bE \log \left(\frac{M(\what{P})}{\sqrt{n}}\right) = -\frac{\gamma}{2}
	\end{equation} 
	where $\what{P}:= \max\{|P|,n^{-1}\}$ and $M(\what{P})$ is defined via the logarithmic integral in \eqref{eq:def_mahler_measure_logaritmic_integral}. See also the introduction in \cite{choi_erdelyi} for a basic survey on the Mahler measure of polynomials with integer coefficients. Our Lemma \ref{lemma:mean_of_log_mahler_measure} (the case $r=1$ and $\ell=1$) shows that the limit \eqref{eq:limit_of_mahler_measure_truncation} remains true when we replace the truncated polynomial $\what{P}$ by $P$. In fact, we can say even more, as Lemma \ref{lemma:mean_of_log_mahler_measure} (the case $r=1$) together with Chebyshev's inequality implies that
	\begin{equation*}
		\frac{M(P)}{\sqrt{n}} \xrightarrow{p} e^{-\gamma/2} \quad \text{as } n\to\infty.
	\end{equation*}
	In view of the above, we hope that Lemma \ref{lemma:mean_of_log_mahler_measure} is of independent interest to some of the readers and will shed some light on the Mahler measure of random polynomials.
	
	The paper is organized as follows. In Section \ref{sec:proof_of_main_thm} we show how the proof of Theorem \ref{thm:main_theorem} follows from Lemma \ref{lemma:mean_of_log_mahler_measure}. Then, in Section \ref{sec:concentration_of_log_integral}, we prove Lemma \ref{lemma:mean_of_log_mahler_measure} with aid of some technicalities that we later prove in Sections \ref{sec:proof_of_prop} and \ref{sec:proof_of_clt}. 
	\subsection*{Notation}
	Throughout we will denote by $\dd\mu$ the normalized measure on $[-\pi,\pi]$ and by $\dd m(z)$ the Lebesgue measure in the plane $\bC = \bR^2$. We write that $f_n\lesssim g_n$ or $f_n = \mathcal{O}(g_n)$ if there exist a constant $C>0$ such that $f_n \leq C \cdot g_n$ for $n$ large enough. We write $f_n = o(g_n)$ if $f_n/g_n \xrightarrow{n\to\infty} 0$, and write $f_n\sim g_n$ if $f_n/g_n \xrightarrow{n\to\infty} 1$. Finally, we denote by $\text{Id}$ the identity matrix (the dimension should be clear from the context).
	
	\subsection*{Acknowledgments}
	I thank my advisors Alon Nishry and Mikhail Sodin for many helpful discussions and for suggestions regarding the presentation of this result. I also wish to thank Aron Wennman and Ofer Zeitouni for fruitful conversations.
	
	\section{Proof of main result}
	\label{sec:proof_of_main_thm}
	We will prove Theorem \ref{thm:main_theorem} by showing the corresponding upper and lower bounds. The main result will follow once we show that
	\begin{equation}
		\label{eq:upper_bound_for_roots}
		\lim_{n\to \infty} \bP\left(\nu_{n}(\bD) \ge \frac{n}{2} + n^{9/10} \right) = 0
	\end{equation}
	and that,
	\begin{equation}
	\label{eq:lower_bound_for_roots}
	\lim_{n\to \infty} \bP\left(\nu_{n}(\bD) \le \frac{n}{2} - n^{9/10} \right) = 0.
	\end{equation}
	As noted in the introduction, we will assume here that Lemma \ref{lemma:mean_of_log_mahler_measure} holds and show how Theorem \ref{thm:main_theorem} follows from it. We will prove Lemma \ref{lemma:mean_of_log_mahler_measure} in Section \ref{sec:concentration_of_log_integral}. Recall that Jensen's formula (see for instance \cite[p.~207]{ahlfors}) states that
	\begin{equation}
		\label{eq:jensen_formula}
		\int_{0}^{r} \frac{N_f(t)}{t} \, \dd t = \int_{-\pi}^{\pi} \log|f(re^{i\theta})| \, \dd\mu(\theta) - \log|f(0)|
	\end{equation}
	where $f$ is an analytic function on $\{|z|\leq r\}$ such that $f(0)\not=0$ and \[
	N_f(t):=\#\{|z|\leq t : \, f(z)=0 \}.\]
	\begin{proof}[Proof of Theorem \ref{thm:main_theorem}]
		We will first prove the upper bound \eqref{eq:upper_bound_for_roots}. Let $\tau = n^{-11/10}$. By Jensen's formula \eqref{eq:jensen_formula} we have
		\begin{align}
		\nonumber \nu_{n}(\bD) &\leq \frac{1}{\log(1+\tau)} \int_{1}^{1+\tau} \frac{\nu_{n}(r\bD)}{r} \dd r \\ \nonumber &= \frac{1}{\log(1+\tau)} \left( \int_{-\pi}^{\pi} \log\left|P((1+\tau)e^{i\theta})\right| \dd \mu(\theta) - \int_{-\pi}^{\pi} \log\left|P(e^{i\theta})\right| \dd \mu(\theta) \right) \\ \nonumber  &= \frac{\log \sigma(1+\tau) - \log \sigma(1)}{\log(1+\tau)} \\ \label{eq:upper_bound_for_zeros_in_small_strip} & \qquad + \frac{1}{\log(1+\tau)} \left( \int_{-\pi}^{\pi} \log|\ti{P}((1+\tau)e^{i\theta})| \dd\mu(\theta) - \int_{-\pi}^{\pi} \log|\ti{P}(e^{i\theta})| \dd\mu(\theta) \right).
		\end{align}
		We differentiate relation \eqref{eq:def_radial_variance} and see that
		\begin{align*}
			&\frac{\partial}{\partial r} \log \sigma(r) = \frac{1}{2} \frac{\sum_{k=0}^{n-1} (2k)r^{2k-1}}{\sum_{k=0}^{n-1} r^{2k}} \\ & \frac{\partial^2}{\partial r^2} \log \sigma(r) = \frac{1}{2} \frac{\sum_{k=0}^{n-1}2k(2k-1)r^{2k-2} }{\sum_{k=0}^{n-1}r^{2k} } - \frac{\left(\sum_{k=0}^{n-1}2kr^{2k-1}\right)^2 }{\left(\sum_{k=0}^{n-1}r^{2k}\right)^2 }.
		\end{align*}
		This implies that $\frac{\partial}{\partial r} \log \sigma(r)\Big|_{r=1} = (n-1)/2$ and that
		\begin{align*}
		\left| \frac{\partial^2}{\partial r^2} \log \sigma(r) \right| &\leq \left| \frac{\sum_{k=0}^{n-1}k(2k-1)r^{2k-2} }{\sum_{k=0}^{n-1}r^{2k} }\right| +\left| \frac{\left(\sum_{k=0}^{n-1}2kr^{2k-1}\right)^2 }{\left(\sum_{k=0}^{n-1}r^{2k}\right)^2 } \right| \\ &\leq \frac{\frac{4}{6}n^{3}}{n/2}  + \frac{n^4}{n^2/2} \leq 4n^2
		\end{align*}
		for all $n$ large enough and for all $r\in[1-\tau,1+\tau]$. Plugging the above into a second order Taylor expansion, while using the fact that $\log(1+\tau)\sim \tau$ for $\tau$ small, yields that
		\begin{equation}
		\label{eq:taylor_expantion_for_variance}
		\left|\frac{\log \sigma(1+\tau) - \log \sigma(1)}{\log(1+\tau)} - \frac{n}{2} \right| \leq  2 \tau n^2.
		\end{equation}
		Altogether, we combine \eqref{eq:upper_bound_for_zeros_in_small_strip} and \eqref{eq:taylor_expantion_for_variance} and observe that, for large enough $n$,
		\begin{align*}
		\bP&\left(\nu_{n}(\bD) \ge \frac{n}{2} + n^{9/10} \right) \\ &\leq \bP\left( \int_{-\pi}^{\pi} \log|\ti{P}((1+\tau)e^{i\theta})| \dd\mu(\theta) - \int_{-\pi}^{\pi} \log|\ti{P}(e^{i\theta})| \dd\mu(\theta)  \geq \tau n^{9/10}  + 2\tau^2 n^2 \right) \\ &= \bP \left( \int_{-\pi}^{\pi} \log|\ti{P}((1+\tau)e^{i\theta})| \dd\mu(\theta) - \int_{-\pi}^{\pi} \log|\ti{P}(e^{i\theta})| \dd\mu(\theta)  \geq 3 n^{-1/5} \right) \\ & \leq \bP\left(\left|\int_{-\pi}^{\pi} \log|\ti{P}((1+\tau)e^{i\theta})| \dd\mu(\theta) + \frac{\gamma}{2} \right| \ge n^{-1/5}\right) \\ & \qquad+ \bP\left(\left|\int_{-\pi}^{\pi} \log|\ti{P}(e^{i\theta})| \dd\mu(\theta) + \frac{\gamma}{2} \right| \ge n^{-1/5}\right) \lesssim n^{2/5-1/2} (\log n)^2 \xrightarrow{n\to\infty} 0
		\end{align*}
		where in the last inequality we applied Chebyshev together with Lemma \ref{lemma:mean_of_log_mahler_measure}. We have thus proved the upper bound \eqref{eq:upper_bound_for_roots}. The lower bound \eqref{eq:lower_bound_for_roots} proceeds in a similar way, starting with the inequality
		\begin{align*}
			\nu_{n}(\bD) &\geq \frac{1}{-\log(1-\tau)} \int_{1-\tau}^{1} \frac{\nu_{n}(r\bD)}{r} \dd r \\ &= \frac{1}{-\log(1-\tau)} \left( \int_{-\pi}^{\pi} \log\left|P(e^{i\theta})\right| \dd \mu(\theta) - \int_{-\pi}^{\pi} \log\left|P((1-\tau)e^{i\theta})\right| \dd \mu(\theta) \right)
		\end{align*}
		and arguing in the same way as above. Altogether \eqref{eq:upper_bound_for_roots} and \eqref{eq:lower_bound_for_roots} gives the proof of Theorem \ref{thm:main_theorem}.
	\end{proof}
	\begin{remark*}
		There is an alternative way of showing that the lower bound \eqref{eq:lower_bound_for_roots} holds once we have proved the upper bound \eqref{eq:upper_bound_for_roots}. An immediate corollary of a result by Konyagin and Schlag \cite[Theorem~1.2]{konyagin_schlag} is that
		\begin{equation}
			\label{eq:no_unimodular_zeros}
			\lim_{n\to\infty}\bP\left(\nu_n(\{|z|=1\}) > 0 \right) = 0.
		\end{equation}
		Let $Q(z):= z^{n-1}Q(1/z)$ and observe that $P(a)=0$ if and only if $Q(1/a)=0$. Since the coefficients of $P$ are i.i.d. random variables, $P$ and $Q$ have the same distribution and the lower bound $\eqref{eq:lower_bound_for_roots}$ follows from \eqref{eq:upper_bound_for_roots} and \eqref{eq:no_unimodular_zeros}.
	\end{remark*}	
	\section{Concentration of the Logarithmic integral}
	\label{sec:concentration_of_log_integral}
	In this section we prove Lemma \ref{lemma:mean_of_log_mahler_measure}. Throughout this section $r\in [1-n^{-11/10} , 1+n^{-11/10}]$ and all constants will be uniform in $r$. We note that it is evident from \eqref{eq:def_radial_variance} that $\sigma(r)^2 = n + \mathcal{O}(n^{9/10})$ for all $r\in [1-n^{-11/10},1+n^{-11/10}]$.
	
	As noted in the introduction, the main technical difficulty in proving Lemma \ref{lemma:mean_of_log_mahler_measure} is the singularity of the logarithm near zero, which in turn causes problems when we want to apply the dominated convergence theorem. We deal with this problem through the following proposition.
	\begin{proposition}
		\label{proposition:bound_on_average_small_ball_probability}
		There exist a constant $C>0$ such that for all $a\in (0,\frac{1}{3})$ and for all $r\in [1-n^{-11/10},1+n^{-11/10}]$ we have
		\[
		\int_{-\pi}^{\pi} \bP\left(|P(re^{i\theta})| \leq a\right) \dd\mu(\theta) \leq C\left(\frac{1}{n^5} + n^{240}a \log(a^{-1})\right).
		\]
	\end{proposition}
	Proposition \ref{proposition:bound_on_average_small_ball_probability} is essentially borrowed from \cite{ibragimov_zeitouni} and, for the reader's convenience, we give its proof in Section \ref{sec:proof_of_prop}. 
	\begin{remark*}
		It should be mentioned here that one could go through a different route to prove uniform integrability of the logarithmic integral. Actually, the much more general (and deeper) log-integrability result of Nazarov, Nishry and Sodin \cite[Corallary~1.2]{nazarov_nishry_sodin} will be enough for our purpose. Still, we provide a proof that does not use \cite{nazarov_nishry_sodin} to keep this note as self-contained as possible. 
	\end{remark*}
	
	For every fixed $\theta\in [-\pi,\pi]$ denote by
	\[
	F_{n}^{\theta}(x):= \bP\left(|\ti{P}(re^{i\theta})|^2\leq x\right),\quad F(x):= 1- e^{-x}.
	\]
	Further, for different values of $\theta \not=\phi \in [-\pi,\pi]$ we denote by
	\[
	F_{n}^{\theta,\phi}(x,y):= \bP\left(|\ti{P}(re^{i\theta})|^2\leq x,\ |\ti{P}(re^{i\phi})|^2\leq y\right).
	\]
	\begin{lemma}
		\label{lemma:central_limit_theorem}
		There exist an absolute constant $C>0$ such that:
		\begin{enumerate}
			\item for all $|\theta| \ge n^{-1/2}$ and for all $x\in (0,\infty)$
			\[
			\left|F_{n}^{\theta}(x) - F(x)\right| \leq \frac{C}{\sqrt{n}}.
			\]
			\item for all $\theta,\phi\not\in [-n^{-1/2},n^{-1/2}]$ such that $|\theta-\phi|>n^{-1/2}$
			\[
			\left|F_{n}^{\theta,\phi}(x,y) - F(x)F(y)\right| \leq \frac{C}{\sqrt{n}}
			\]
			for all $x,y\in (0,\infty)$.
		\end{enumerate}
	\end{lemma}
	Lemma \ref{lemma:central_limit_theorem} is a simple consequence of a Central Limit Theorem for a sum of independent (but not identically distributed) random vectors, with the usual Berry-Esseen bound on the error. It is not particularly interesting, and we postpone the proof to Section \ref{sec:proof_of_clt}.
	
	\begin{proof}[Proof of Lemma \ref{lemma:mean_of_log_mahler_measure}]
		For concreteness and brevity we only prove the case $\ell=2$ (the case $\ell=1$ is simpler). Denote for the proof $p(\theta) := \ti{P}(re^{i\theta})$, and consider the event
		\[
		\cA_{\theta} := \left\{|P(re^{i\theta})| \leq n^{-A}  \right\}
		\] 
		for some large absolute constant $A \ge 1$ that we choose later. Since the function $\log(x)$ is locally integrable, we may apply Fubini and see that
		\begin{align*}
			\bE&\left[\left(\int_{-\pi}^{\pi} \log |\ti{P}(re^{i\theta})| \dd\mu(\theta) \right)^2\right] \\ &= \bE\left[\iint_{[-\pi,\pi]^2} \log|p(\theta)| \log|p(\phi)| \dd\mu(\theta) \dd\mu(\phi)\right] \\ &= \bE\left[\iint_{[-\pi,\pi]^2} L(\theta,\phi) \, \mathbf{1}_{\cA_\theta^{c}}\mathbf{1}_{\cA_\phi^{c}} \, \dd\mu(\theta) \dd\mu(\phi)\right] + 2\bE\left[\iint_{[-\pi,\pi]^2} L(\theta,\phi) \, \mathbf{1}_{\cA_\theta}\mathbf{1}_{\cA_\phi^{c}} \, \dd\mu(\theta) \dd\mu(\phi)\right] \\ & \qquad + \bE\left[\iint_{[-\pi,\pi]^2} L(\theta,\phi) \, \mathbf{1}_{\cA_\theta}\mathbf{1}_{\cA_\phi} \, \dd\mu(\theta) \dd\mu(\phi)\right] =: E_1 + 2E_2 + E_3
		\end{align*}
		where $L(\theta,\phi) := \log|p(\theta)| \log|p(\phi)|$. We first show that $E_2$ and $E_3$ are small, and then turn to prove that the main term coming from $E_1$ is as claimed. 
		
		Note that since $P(z) = \prod_{j=1}^{n}(z-\beta_j)$ has $n$ zeros with $|\beta_j|\in \left[\frac{1}{2},2\right]$ for all $j=1,\ldots,n$ we have the simple bound 
		\begin{equation}
			\label{eq:naive_bound_on_integral_log}
			\int_{-\pi}^{\pi} \left|\log|p(\theta)|\right|^{4}\dd\mu(\theta) \lesssim  \sum_{j_1,\ldots,j_4=1}^{n} \left(\int_{-\pi}^{\pi} \prod_{\ell=1}^{4} \left|\log| e^{i\theta} - \beta_{j_\ell}|\right| \dd \mu(\theta)\right) \lesssim n^4
		\end{equation}
		where the last inequality is true since there are $n^{4}$ terms in the sum, all are uniformly bounded since $|\beta_j|\in \left[\frac{1}{2},2\right]$.
		
		By applying Tonelli's theorem, Cauchy-Schwarz inequality (twice) and using \eqref{eq:naive_bound_on_integral_log} we get
		\begin{align*}
			|E_3| &\leq \iint_{[-\pi,\pi]^2} \bE \left[ \left|\log|p(\theta)|\right| \left|\log|p(\phi)|\right| \mathbf{1}_{\cA_{\theta}} \mathbf{1}_{\cA_{\phi}} \right] \dd\mu(\theta) \dd\mu(\phi) \\ &\leq \left(\int_{-\pi}^{\pi} \sqrt{\bE\left[\log^2|p(\theta)|\mathbf{1}_{\cA_{\theta}}\right]} \dd \mu(\theta) \right)^2 \\ & \leq\int_{-\pi}^{\pi} \sqrt{\bE\left[\log^4|p(\theta)|\right] \bP\left(\cA_{\theta}\right)} \dd\mu(\theta) \lesssim n^2 \times \left(\int_{-\pi}^{\pi} \bP\left(\cA_\theta\right) \dd \mu(\theta)\right)^{1/2} \lesssim \frac{1}{\sqrt{n}}
		\end{align*}
		where the last inequality is Proposition \ref{proposition:bound_on_average_small_ball_probability} applied with $a=n^{-A}$ and $A$ large enough; $A\ge 250$ is good for our purpose.
		
		Similarly to the above we can show that $|E_2| \lesssim n^{-1/2}$ so it remains to show that $E_1$ gives the main term, this we do in what follows.
	
		Let $(\theta,\phi)$ be coordinates in the square $[-\pi,\pi]^2$ and consider the subset
		\begin{align*}
			D_n:= \left\{|\theta- \phi| \leq n^{-1/2} \right\}\cup \{|\theta|\leq n^{-1/2} \} \cup \{|\phi|\leq n^{-1/2} \} \subset [-\pi,\pi]^2.
		\end{align*} 
		Since the Lebesgue measure of $D_n$ is small we have
		\begin{align*}
			\iint_{D_n} |\log|p(\theta)|\log|p(\phi)||&\mathbf{1}_{\cA_\theta^{c}}\mathbf{1}_{\cA_\phi^{c}} \dd\mu(\theta) d\mu(\phi) \\ &\lesssim_{A} (\log n)^2 m(D_n)\lesssim_A  \frac{(\log n)^2}{\sqrt{n}}
		\end{align*}
		so the proof of the lemma will follow once we show that 
		\begin{equation}
			\label{eq:expectation_of_log_outside_bad_set}
			\bE\left[\log|p(\theta)|\log|p(\phi)|\mathbf{1}_{\cA_\theta^{c}}\mathbf{1}_{\cA_\phi^{c}}\right] = \frac{\gamma^2}{4} + \mathcal{O}\left(\frac{\log^2 n}{\sqrt{n}}\right)
		\end{equation}
		uniformly in $(\theta,\phi)\in [-\pi,\pi]^2 \setminus D_n$. It is routine to check that
		\begin{align*}
			\int_{n^{-A}}^{n} (\log x) \dd F(x) &= \int_{0}^{\infty}(\log x) e^{-x} \dd x + \mathcal{O}_{A}\left(\frac{\log n}{n^{A}}\right) \\ &= -\gamma + \mathcal{O}_{A}\left(\frac{\log n}{n^{A}}\right).
		\end{align*}
		Furthermore, since $(\theta,\phi)\not\in D_n$ we may apply item 2 in Lemma \ref{lemma:central_limit_theorem} (for the case $\ell=1$ we need to apply item 1) and see that
		\begin{align*}
			&\left|\bE\left[\log|p(\theta)|^2\log|p(\phi)|^2\mathbf{1}_{\cA_\theta^{c}}\mathbf{1}_{\cA_\phi^{c}}\right] - \gamma^2\right| \\ &\leq \int_{n^{-A}}^{n} \int_{n^{-A}}^{n} |\log(x) \log(y)| \dd \left(F_n^{\theta,\phi}(x,y) - F(x)F(y) \right) + \mathcal{O}_{A}\left(\frac{\log n}{n^{A}}\right) \\ & \lesssim \int_{n^{-A}}^{n} \int_{n^{-A}}^{n} \frac{|F_n^{\theta,\phi}(x,y) - F(x)F(y)|}{xy} \dd x \dd y + \mathcal{O}_{A}\left(\frac{\log n}{n^{A}}\right) \lesssim_{A} \frac{(\log n)^2}{\sqrt{n}} 
		\end{align*}
		and hence we get \eqref{eq:expectation_of_log_outside_bad_set} and finish the proof.
	\end{proof}

	\section{Proof of Proposition \ref{proposition:bound_on_average_small_ball_probability}}
	\label{sec:proof_of_prop}
	In this section we prove Proposition \ref{proposition:bound_on_average_small_ball_probability}. For this, we will make use of the classical Tur\`{a}n's Lemma, quoted here for the convenience of the reader.
	\begin{lemma}[{\cite[Chapter~5.3]{montgomery}, \cite[Section~1.1]{nazarov}}]
		\label{lemma:turan_lemma}
		Let $T(x) := \sum_{k=0}^{h} c_k e^{im_k x}$ be a trigonometric polynomial with $c_k\in \bC$ and $m_1 < m_2 < \ldots < m_h\in \bZ$. Then for any interval $E\subset [-\pi,\pi]$ we have
		\[
		\sup_{x\in [-\pi,\pi]} |T(x)| \leq \left(\frac{12}{\mu(E)}\right)^{h-1} \sup_{x\in E} |T(x)|.
		\]
	\end{lemma}
	\begin{proof}[Proof of Proposition \ref{proposition:bound_on_average_small_ball_probability}]
		Our starting point is Esseen's concentration inequality, see \cite[Theorem~7.17]{tao_vu}. It states that
		\begin{equation}
			\label{eq:esseen_concent_ineq}
			\bP\left(|P(re^{i\theta})| \leq a\right) \leq \bP\left(|\re P(re^{i\theta})| \leq a\right) \lesssim a \int_{-a^{-1}}^{a^{-1}} |g(\lambda)| \dd\lambda
		\end{equation}
		where $g$ is the characteristic function of $\re (P(re^{i\theta})) = \sum_{k=0}^{n-1} X_k r^k \cos(k\theta)$. Denote by $\lVert x\rVert_{\bR/\bZ}:= \min\{|x-m|: m\in \bZ\}$. By using the inequalities $|\cos x|\leq \exp(-2\lVert x/\pi \rVert_{\bR/\bZ}^2)$ and $\cos(2x) \ge 1 - 2\pi^2 \lVert x/\pi \rVert_{\bR/\bZ}^2$ valid for all $x\in \bR$ we get
		\begin{align*}
			|g(\lambda)| &= \left|\prod_{k=0}^{n-1} \cos\left(\lambda r^k \cos(k\theta) \right)\right| \\ & \leq \exp\left(-2\sum_{k=0}^{n-1} \lVert \lambda r^k \cos(k\theta)/\pi \rVert_{\bR/\bZ}^2 \right) \leq e^{-n/\pi^2} \exp\left( \frac{1}{\pi^2} \sum_{k=0}^{n-1} \cos(2\lambda r^k\cos(k\theta))\right).
		\end{align*} 
		Combining this with \eqref{eq:esseen_concent_ineq} we use Fubini and get
		\begin{align}
			\label{eq:integral_of_esseen_inequality}
			\int_{-\pi}^{\pi} &\bP\left(|P(re^{i\theta})| \leq a\right) \dd \mu(\theta) \\ \nonumber & \lesssim a \int_{-a^{-1}}^{a^{-1}} e^{-n/\pi^2}\int_{-\pi}^{\pi} \exp\left( \frac{1}{\pi^2} \sum_{k=0}^{n-1} \cos(2\lambda r^k\cos(k\theta))\right) \dd \mu(\theta) \dd \lambda.
		\end{align}
		Fix $\lambda\in[-a^{-1},a^{-1}]$ and denote by
		\begin{align*}
			&\Lambda_1(\lambda,n) := \left\{\theta\in[-\pi,\pi] \mid \sum_{k=0}^{n-1} \cos(2\lambda r^k\cos(k\theta)) < \frac{n}{2} \right\}, \\ &\Lambda_2(\lambda,n):= [-\pi,\pi] \setminus \Lambda_1(\lambda,n).
		\end{align*}
		For all $\theta\in\Lambda_1(\lambda,n)$ the inequality \eqref{eq:integral_of_esseen_inequality} gives an exponential bound in $n$, so we are left to estimate the measure of $\Lambda_2(\lambda,n)$ on the circle. Indeed, denote for the moment by $a_k(\theta):= 2r^{k} \cos(k\theta)$ and note that
		\begin{align}
			\label{eq:opening_brackets_sum_of_cosines}
			\int_{-\pi}^{\pi} \left|\sum_{k=0}^{n-1} \cos(\lambda a_k(\theta)) \right|^{10} \dd \mu(\theta) &\leq \int_{-\pi}^{\pi} \left|\sum_{k=0}^{n-1} e^{i\lambda a_k(\theta)} + e^{-i\lambda a_k(\theta)} \right|^{10} \dd \mu(\theta) \nonumber \\ &= \sum_{s_1,\ldots,s_{10}\in\{\pm 1\}} \sum_{k_1,\ldots k_{10}=0}^{n-1} \int_{-\pi}^{\pi} e^{i\lambda \sum_{\ell=1}^{10} s_\ell a_{k_\ell}(\theta) } \dd \mu(\theta).
		\end{align}
		Fix a tuple $(k_1,\ldots,k_{10})$ in the above sum \eqref{eq:opening_brackets_sum_of_cosines}. For all $\ell\in\{1,\ldots,10\}$ let $m_\ell := \# \left\{ j\in\{1,\ldots,10\} \ \text{such that } k_j = k_\ell \right\}$. For all tuples $(k_1,\ldots,k_{10})$ for which all $\{m_\ell\}_{\ell=1}^{10}$ are even it may happen that $\sum_{\ell=1}^{10} s_\ell a_{k_\ell}(\theta) \equiv 0$ for all $\theta\in[-\pi,\pi]$ for some particular choice of signs $s_1,\ldots,s_{10}$. Still, there are at most $\mathcal{O}(n^5)$ such tuples of $\{k_1,\ldots,k_{10} \}$ in \eqref{eq:opening_brackets_sum_of_cosines}. To deal with the rest of the sum, we have the following claim.
		\begin{claim}
			\label{claim:stationary_phase_for_sparse_trig_poly}
			Let $T(\theta):= \sum_{\ell=1}^{h} b_\ell e^{i k_{\ell}\theta}$ be a non-constant trigonometric polynomial with $b_\ell\in[-1,-\frac{1}{2}]\cup[\frac{1}{2},1]$ and $k_\ell\in \{-n,\ldots,n\}$ for all $\ell=1,\ldots,h$. Then, for all $c>0$,
			\[
			\left|\int_{-\pi}^{\pi} e^{i\lambda T(\theta) } \dd \mu(\theta)\right| \lesssim_{h} \frac{1}{n^{c}} + \frac{n^{2h(c+2)}}{\lambda}.
			\]
		\end{claim}
		Assuming Claim \ref{claim:stationary_phase_for_sparse_trig_poly} for the moment, we plug into \eqref{eq:opening_brackets_sum_of_cosines} with
		\[
		T(\theta) := \sum_{\ell=1}^{10} s_\ell a_{k_\ell}(\theta) = \sum_{\ell=1}^{10} s_\ell r^{k_{\ell}} \left(e^{ik_\ell \theta} + e^{-ik_\ell \theta} \right)
		\]
		and $c=5$ to obtain
		\begin{align*}
			\int_{-\pi}^{\pi} \left|\sum_{k=0}^{n-1} \cos(\lambda a_k(\theta)) \right|^{10} \dd \mu(\theta) &\lesssim n^{5} + n^{10}\times \left(\frac{1}{n^5} + \frac{n^{240}}{\lambda}\right) \\ &\lesssim n^{5} + \frac{n^{250}}{\lambda}.
		\end{align*}
		By applying Markov's inequality we get that
		\begin{align*}
			\mu\left( \Lambda_2(\lambda,n) \right) \lesssim \min\left\{\frac{n^5 + n^{250}/\lambda}{n^{10}} ,1 \right\} = \min\left\{\frac{1}{n^5} + \frac{n^{240}}{\lambda} , 1\right\}.
		\end{align*}
		Plugging the above into \eqref{eq:integral_of_esseen_inequality} gives
		\begin{align*}
			\int_{-\pi}^{\pi} \bP\left(|P(re^{i\theta})| \leq a\right) \dd \mu(\theta) & \lesssim e^{-n/2\pi^2} + a\int_{-a^{-1}}^{a^{-1}}\mu\left(\Lambda_2(\lambda,n)\right) \dd \lambda\\ & \lesssim  e^{-n/2\pi^2} + a + a\int_{1}^{a^{-1}}\mu\left(\Lambda_2(\lambda,n)\right) \dd \lambda \\ &\lesssim e^{-n/2\pi^2} + a + \frac{1}{n^5} + n^{240}a \log(a^{-1}).
		\end{align*}
	\end{proof}
	\begin{proof}[Proof of Claim \ref{claim:stationary_phase_for_sparse_trig_poly}]
		Let $\theta_1< \theta_2 < \ldots < \theta_L \in [-\pi,\pi]$ be all points on the unit circle which are solutions to the equation
		\[
		|T^\prime(\theta_j)| = n^{-h(c+1)}, \qquad j=1,\ldots,L.
		\] 
		Since $T^\prime$ is a trigonometric polynomial of degree at most $n$, we have $L\leq 2n$. We split the unit circle into $L$ intervals given by
		\[
		I_j:=[\theta_j,\theta_{j+1}] \quad \text{for } j=1,\ldots L-1 \quad  \text{and} \quad I_L:= [-\pi,\pi] \setminus \left(\bigcup_{j=1}^{L-1} I_j\right).
		\]
		We say that $I_j$ is \emph{good} if $|T^\prime(\theta)| \geq n^{-h(c+1)}$ for all $\theta\in I_j$. We say that $I_j$ is \emph{bad} if it is not good. Denote by $\mathcal{G}$ the collection of good intervals and by $\mathcal{B}$ the collection of bad intervals. First, we claim that for each $I\in \mathcal{B}$ we must have $\mu(I) \leq n^{-(c+1)}$. Assume otherwise, then by our construction there exist an interval $I\subset [-\pi,\pi]$ with
		\[
		\mu(I) > \frac{1}{n^{c+1}},\quad \text{and } \ \sup_{\theta\in I}|T^\prime(\theta)| \leq n^{-h(c+1)}.
		\] 
		Tur\`{a}n's Lemma (Lemma \ref{lemma:turan_lemma}) yields that
		\[
		\sup_{\theta\in [-\pi,\pi]} |T^\prime(\theta)| \leq \left(\frac{12}{\mu(I)}\right)^{h-1} n^{-h(c+1)} \lesssim_{h} n^{(c+1)(h-1)} n^{-h(c+1)} = n^{-1-c}  
		\]
		which clearly contradicts the fact that $\int_{-\pi}^{\pi} |T^{\prime}|^2 \dd \mu \ge 1/4$ as our polynomial $T$ is non-constant and $k_\ell$ are integers. Altogether we get that
		\[
		\mu\left(\bigcup_{I\in \mathcal{B}} I\right) \leq \frac{2n}{n^{1+c}} \lesssim \frac{1}{n^c}.
		\]
		To conclude the claim, it remains to integrate by parts on the good intervals and see that
		\begin{align*}
		\left|\int_{-\pi}^{\pi} e^{i\lambda T(\theta)} \dd \mu(\theta)\right| & \leq \frac{1}{n^c} + \sum_{I\in \mathcal{G}} \left|\int_{I} e^{i\lambda T(\theta)} \dd \mu(\theta) \right| \\ & \lesssim \frac{1}{n^c} +  \frac{n^{h(c+1)}}{\lambda} + \sum_{I\in \mathcal{G}}\frac{1}{\lambda}\left|\int_{I} e^{i\lambda T(\theta)} \frac{T^{\prime\prime}(\theta)}{(T^\prime(\theta))^2} \dd \mu(\theta) \right| \lesssim \frac{1}{n^c} + \frac{n^{2h(c+2)}}{\lambda}
		\end{align*}
		as claimed.
	\end{proof}
	\section{Proof of Lemma \ref{lemma:central_limit_theorem}}
	\label{sec:proof_of_clt}
	Throughout this section we fix the parameters $r,\theta,\phi$ as prescribed in the statement of Lemma \ref{lemma:central_limit_theorem}. We will also re-use the notation $p(\theta):= \ti{P}(re^{i\theta})$. As indicated in the previous section, the proof of Lemma \ref{lemma:central_limit_theorem} will follow from a version of Berry-Esseen inequality on the remainder term in the Central Limit theorem, when considering a sum of independent (but not identically distributed) random vectors with uniformly bounded third moment, see \cite[Corllary~17.2]{bha_rao}. Hence, we only need to check that the covariance structure of 
	\begin{equation*}
		p(\theta) = \frac{1}{\sigma(r)}\left(\sum_{k=0}^{n-1} X_k r^k \cos(k\theta), \sum_{k=0}^{n-1} X_k r^k \sin(k\theta) \right) \in \bR^2 
	\end{equation*}
	and of
	\begin{align*}
		&\left(p(\theta),p(\phi)\right) \\ &= \frac{1}{\sigma(r)}\left(\sum_{k=0}^{n-1} X_k r^k \cos(k\theta), \sum_{k=0}^{n-1} X_k r^k \sin(k\theta) ,\sum_{k=0}^{n-1} X_k r^k \cos(k\phi), \sum_{k=0}^{n-1} X_k r^k \sin(k\phi) \right) \in \bR^4
	\end{align*}
	scale like $\frac{1}{2}\text{Id}$ in the respective dimension. 
	
	Denote by
	\begin{equation}
		\label{eq:definition_of_covariance_2_d}
		V(\theta) := \frac{1}{\sigma(r)^2} \sum_{k=0}^{n-1} r^{2k} \begin{bmatrix}
		\cos^2(k\theta) & \cos(k\theta)\sin(k\theta) \\ \cos(k\theta)\sin(k\theta) & \sin^2(k\theta)
		\end{bmatrix},
	\end{equation}
	and by
	\begin{align}
	\nonumber
	V(\theta,\phi) &:= \frac{1}{\sigma(r)^2} \sum_{k=0}^{n-1} r^{2k} \\ \nonumber &\qquad \times \begin{bmatrix}
	\cos^2(k\theta) & \cos(k\theta)\sin(k\theta) & \cos(k\theta)\cos(k\phi) & \cos(k\theta)\sin(k\phi) \\ \cos(k\theta)\sin(k\theta) & \sin^2(k\theta) & \sin(k\theta)\cos(k\phi) & \sin(k\theta)\sin(k\phi) \\ \cos(k\phi)\cos(k\theta) & \cos(k\phi)\sin(k\theta) & \cos^2(k\phi) & \cos(k\phi)\sin(k\phi) \\ \sin(k\phi)\cos(k\theta) & \sin(k\phi)\sin(k\theta) & \sin(k\phi)\cos(k\phi) & \sin^2(k\phi)
	\end{bmatrix}. \\ \label{eq:definition_of_covariance_4_d} 
	\end{align}
	Using \eqref{eq:definition_of_covariance_2_d}, \eqref{eq:definition_of_covariance_4_d} and the trigonometric identities
	\begin{align*}
		\cos(x)\cos(y) = \frac{\cos(x-y) + \cos(x+y)}{2}, \quad   \cos(x) \sin(y) &= \frac{\sin(x+y) - \sin(x-y)}{2} \\ \text{and} \quad \sin(x)\sin(y) &= \frac{\cos(x-y) - \cos(x+y)}{2},
	\end{align*}
	Lemma \ref{lemma:central_limit_theorem} will follow from \cite[Corllary~17.2]{bha_rao} combined with the following simple claim.
	\begin{claim}
		\label{claim:cancellation_in_trigonometric_sums}
		Let $\eta\in [-\pi,\pi] \setminus [-n^{-1/2},n^{-1/2}]$, then
		\begin{align*}
			&\frac{1}{\sigma(r)^2} \left|\sum_{k=0}^{n-1} r^{2k} \sin(k \eta)\right| \lesssim \frac{1}{\sqrt{n}},\qquad \frac{1}{\sigma(r)^2} \left|\sum_{k=0}^{n-1} r^{2k} \cos(k \eta)\right| \lesssim \frac{1}{\sqrt{n}}
		\end{align*}
		uniformly in $r\in[1-n^{-11/10} , 1+n^{-11/10}]$.
	\end{claim}
 	\begin{proof}
 		Recall from \eqref{eq:def_radial_variance} that $\sigma(r)^2 = n + \mathcal{O}(n^{9/10})$ uniformly in $r$. Both sums are bounded by
 		\begin{align*}
 			\frac{1}{\sigma(r)^2} \left|\sum_{k=0}^{n-1}r^{2k} e^{ik\eta} \right| &= \frac{1}{\sigma(r)^2} \left|\frac{1-r^{2n}e^{in\eta}}{1-r^2e^{i\eta}}\right| \\ & \lesssim \frac{1}{n} \frac{\left|1-e^{in\eta}\right| + \left|1-r^{2n}\right|}{|1-e^{i\eta}| - |1-r^2|} \lesssim \frac{1}{n} \frac{2 + n^{-1/10}}{|\eta| - n^{-11/10} } \lesssim  \frac{1}{\sqrt{n}}.
 		\end{align*}
 	\end{proof}
 	\begin{proof}[Proof of Lemma \ref{lemma:central_limit_theorem}]
 		Again, we prove only item 2, and leave (the simpler) item 1 to the reader. By \cite[Corollary~17.2]{bha_rao}, we have that 
 		\begin{align*}
 			\sup_{x,y\ge 0}\left|F_{n}^{\theta,\phi}(x,y) - G_n^{\theta,\phi}(x,y)\right| \lesssim \frac{1}{\sqrt{n}},
 		\end{align*}
 		where $G_n^{\theta,\phi}$ the CDF of a mean-zero Gaussian vector in $\bR^4$ with covariance structure $V = V(\theta,\phi)$ given by \eqref{eq:definition_of_covariance_4_d}. And so, it remains to show that
 		\begin{equation}
 			\label{eq:comparison_inequality_gaussians}
 			\sup_{x,y\ge 0}\left|F(x)F(y) - G_n^{\theta,\phi}(x,y)\right| \lesssim \frac{1}{\sqrt{n}}.
 		\end{equation}
 		For $\xi \in \bR^4$ and $\Sigma$ a $4\times4$ positive-definite matrix we set 
 		\begin{equation*}
 			f_{\Sigma}(\xi):= \frac{1}{4\pi^{2} \sqrt{\det(\Sigma)}} \exp\left(-\frac{1}{2} \xi^{T} \Sigma^{-1}\xi\right),\quad \text{and  } f(\xi):= f_{\frac{1}{2}\text{Id}}(\xi).
 		\end{equation*}
 		That is, $f_{\Sigma}$ is the density function (with respect to the Lebesgue measure $\dd m_4$ in $\bR^4$) of a mean-zero Gaussian vector in $\bR^4$ with covariance matrix $\Sigma$. By writing $\xi = (z,w)$ with $z,w\in \bR^2$ we use Claim \ref{claim:cancellation_in_trigonometric_sums} and get that
 		\begin{align*}
 			\sup_{x,y\ge 0}&\left|F(x)F(y) - G_n^{\theta,\phi}(x,y)\right| \\ & \leq \int_{\bR^4} \left|f(\xi) - f_{V}(\xi) \right| \dd m_4(\xi)  \\ &\leq \left(\int_{ \substack{|w| > n^{1/4} \\ z\in \bR^2}} + \int_{\substack{|z| > n^{1/4} \\ w\in \bR^2} } + \int_{ \substack{|z| \leq n^{1/4} \\ w \leq n^{1/4}}}\right)\left|f(\xi) - f_{V}(\xi) \right| \dd m_4(\xi) \\ & \lesssim e^{-\sqrt{n}} + \int_{ \substack{|z| \leq n^{1/4} \\ w \leq n^{1/4}}} e^{-|z|^2-|w|^2} \left|1 - \exp\left(-\frac{1}{2}(z,w)^{T} \left(V^{-1} - \text{Id}\right)(z,w)\right)\right| \dd m(z) \dd m(w) \\ &\lesssim e^{-\sqrt{n}} + \frac{1}{\sqrt{n}}\int_{\bR^4} |\xi|^2 e^{-|\xi|^{2}} \dd m_4(\xi) \lesssim \frac{1}{\sqrt{n}}
 		\end{align*}
 		which proves \eqref{eq:comparison_inequality_gaussians} and hence the lemma.
 	\end{proof}

	\vspace{1cm}
	\noindent

\end{document}